\documentclass{amsart}
\usepackage{amsfonts}
\usepackage{amsmath,amssymb}

\newtheorem{theorem}{Theorem}[section]

\newtheorem{corollary}[theorem]{Corollary}

\newtheorem{example}[theorem]{Example}

\newtheorem{lemma}[theorem]{Lemma}

\newtheorem{problem}[theorem]{Problem}
\newtheorem{proposition}[theorem]{Proposition}
\newtheorem{remark}[theorem]{Remark}

\begin{document}
\title{The $X$-semiprimeness of rings}
\author{Grigore C\u{a}lug\u{a}reanu, Tsiu-Kwen Lee, Jerzy Matczuk}
\thanks{Corresponding author:\ Tsiu-Kwen Lee.\eject
\eject
\eject
Orcid: 0000-0002-3353-6958,
0000-0002-1262-1491, 0000-0001-8749-6229}
\address{Department of Mathematics, Babe\c{s}-Bolyai University,
Cluj-Napoca, 400084, Romania}
\email{calu@math.ubbcluj.ro}
\address{Department of Mathematics, National Taiwan University, Taipei 106,
Taiwan}
\email{tklee@math.ntu.edu.tw}
\address{Department of Mathematics, University of Warsaw, Banacha 2, 02-097
Warsaw, Poland}
\email{jmatczuk@mimuw.edu.pl}

\begin{abstract}
For a nonempty subset $X$ of a ring $R$, the ring $R$ is called $X$-semiprime if,
given $a\in R$, $aXa=0$ implies $a=0$. This provides   a proper class of
semiprime rings. First, we clarify the relationship between idempotent semiprime and unit-semiprime rings.
Secondly, given a Lie ideal $L$ of a ring $R$, we offer a
criterion for $R$ to be $L$-semiprime. For a prime ring $R$, we
characterizes Lie ideals $L$ of $R$ such that $R$ is $L$-semiprime. Moreover, $X$-semiprimeness of matrix rings, prime rings (with a nontrivial idempotent), semiprime rings, regular rings, and subdirect products are studied.
\end{abstract}

\maketitle

\section{Introduction}
Throughout the paper, rings are always associative with identity. For a ring
$R$, $U(R)$ (resp. $Id(R)$) denotes the set of all units (resp. idempotents)
of $R$, and $Z(R)$ stands for the center of $R$.

A ring $R$ is called {\it semiprime} if, given $a\in R$, $aRa=0$ implies $%
a=0$. Also, it is called {\it prime} if, given $a,b\in R$, $aRb=0$
implies that either $a=0$ or $b=0$. The purpose of the paper is to study a
more general notion concerning semiprimeness and primeness of rings.\vskip4pt

\noindent\textbf{Definition}. Let $X$ be a nonempty subset of a ring $R$.
The ring $R$ is called {\it $X$-semiprime} (resp. {\it $X$-prime}) if, given $a\in R$, $aXa=0$ implies $a=0$ (resp. if, given $a,b\in R$, $%
aXb=0$ implies either $a=0$ or $b=0$).\vskip4pt

 An $X$-semiprime (resp. $X$-prime) ring with $X=\{1\}$ is reduced (resp. a domain). Also, a reduced ring (resp. domain) is $X$-semiprime (resp. $X$-prime) if $1\in X$.

\begin{proposition}\label{pro5}
Given a subset $X$ of a ring $R$, if $R$ is a prime $X$-semiprime ring, then
it is $X$-prime.
\end{proposition}

\begin{proof}
Indeed, let $aXb=0$, where $a, b\in R$. Then $bxaXbxa=0$
for all $x\in R$. The $X$-semiprimeness of $R$ implies that $bxa=0$ for all $x\in R$. By the primeness of $R$,
either $a=0$ or $b=0$, as desired.
\end{proof}

According to Proposition \ref{pro5}, it suffices to study the $X$-semiprime case.
In this way, we can make our statements more concise.

In \cite{cal2} the case $X=U(R)$ was considered, the ring $R$ is called {\it unit-semiprime} if it is $U(R)$-semiprime.
This turned out to be
an interesting class of semiprime rings. In this paper we first consider the
special case $X=Id(R)$. A ring $R$ is called {\it idempotent semiprime} if it is $Id(R)$-semiprime. Let $X^{+}$ denote
the additive subgroup of $R$ generated by $X$. Clearly, a ring $R$ is $X$%
-semiprime if and only if it is $X^{+}$-semiprime.  Let $E(R):=\mathit{Id}(R)^{+}$. Thus the $E(R)$-semiprimeness of $R$ just means that $R$ is idempotent semiprime.

In this paper we will be concerned with the following\vskip6pt

\begin{problem}\label{problem3}
 Given a semiprime or prime ring $R$, find subsets $X$  of $R$ such that $R$ is $X$-semiprime, that is, given $a\in R$,
$$
aXa=0\ \ \Longrightarrow\ \ a=0.
$$
\end{problem}

We organize the paper as follows.

In \S 2 it is proved that every idempotent semiprime ring is unit-semiprime (see Theorem \ref{thm9}).
The converse is in general not true by an example from A. Smoktunowicz in \cite{smok}.
Applying Theorem \ref{thm9} we give an example of a prime ring, which is not idempotent semiprime (see Theorem \ref{thm24}).

In \S 3 we establish a criterion for Lie ideals $L$ of a semiprime
(resp. prime) ring $R$ such that $R$ is $L$-semiprime  (see Theorem \ref{thm3}) and,
in particular, matrix rings are studied. We also conclude that idempotent
semiprime is not a Morita invariant property.

In \S 4 we give a complete characterization of the $L$-semiprimeness of a prime ring $R$ for any given Lie ideal $L$  (see Theorem \ref{thm8}).
Also, in a prime ring $R$ possessing a nontrivial idempotent, its additive
subgroups $X$, which are invariant under all special automorphisms of $R$, are
characterized by Chuang's theorem and hence $R$ is $X$-prime \cite{C} (see Theorem \ref{thm10}).

In \S 5  we characterize the $d(R)$-semiprimeness of a given prime ring $R$, where $d$ is a derivation of $R$. Moreover, we also study the $d(L)$-semiprimeness of $R$ for $L$ a Lie ideal of $R$. See Theorems \ref{thm21} and \ref{thm23}.

In \S 6 and \S 7 we are concerned with the problem whether, given a Lie ideal $N$ of a semiprime ring $R$, $\ell_R(N)=0$  is a sufficient condition for $R$ to be $N$-semiprime, where $\ell_R(N)$ denotes the left annihilator of $N$ in $R$. It is in general not true. However, it is indeed true if either $N=[E(R), R]$ or $N=[L, R]$ when $R$  is $2$-torsion free and $L$ is a Lie ideal of $R$ (see Theorems \ref{thm22} and \ref{thm13}).

In \S 8  it is proved that every regular ring is idempotent semiprime but is in general not $[E(R), R]$-semiprime (see Theorem \ref{thm2} and Example \ref{example3}).
Finally, in \S 9, we study the subdirect product properties of $X$-semiprime rings.

Whenever it is more convenient, we will use the widely accepted shorthand
\textquotedblleft iff\textquotedblright\ for \textquotedblleft if and only
if\textquotedblright\ in the text.

\section{Idempotent semiprime}

We begin with clarifying the relationship between \textquotedblleft idempotent
semiprime" and \textquotedblleft unit-semiprime".

\begin{theorem}
\label{thm9}
If $R$ is an $E(R)$-semiprime ring, then it is $U(R)$-semiprime.
\end{theorem}

\begin{proof}
Let $a\in R$ be such that $%
aua=0$ for all $u\in U(R)$. In particular, $a^{2}=0$. Let $e=e^{2}\in R$ and
$x\in R$. Since $1+ex(1-e),1+(1-e)xe\in U(R)$, we have
\begin{equation*}
a\big(1+ex(1-e)\big)a=0=a\big(1+(1-e)xe\big)a.
\end{equation*}%
Thus $aeR(1-e)a=0$ and $a(1-e)Rea=0$. The semiprimeness of $R$ implies that $%
(1-e)ae=0=ea(1-e)$ and so $ae=ea$. Hence $aea=ea^{2}=0$, that is, $aE(R)a=0$
and so $a=0$, as desired.
\end{proof}

The converse implication of Theorem \ref{thm9} does not hold.

\begin{example}
\label{example2}\normalfont Let $K$ be a countable field and let $S$ be the
simple nil $K$-algebra constructed by A. Smoktunowicz in \cite{smok}. Let $R$
be the unital algebra obtained by adjoining unity to $S$ using $K$. Then $R$
is a local prime ring with maximal ideal $S$ and every element of $R$ is a
sum of (two) units. Thus $R$ is unit-prime. However, $R$ is not idempotent
prime as it has only trivial idempotents and it is not a domain.
\end{example}

The following offers a prime ring which is
not idempotent semiprime.

\begin{theorem}
\label{thm24}
Let $R=k\left\langle x, y|x^{2}=0\right\rangle $, where $k$ is a field. Then $R$ is a prime ring which is not idempotent
semiprime.
\end{theorem}

\begin{proof}
It is known that $R$ is a prime ring (see \cite[p.92]{coh}). In view of \cite[p.28]{cal2}, $R$ is not unit-semiprime. It follows from Theorem \ref{thm9} that $R$ is not idempotent semiprime.
\end{proof}

 In fact, $Id(R)=\{0, 1\}$ for the ring $R$ in Theorem \ref{thm24}. It is an immediate consequence of Theorem \ref{thm7}.

\section{Lie ideals}

Let $R$ be a ring. Given $x,y\in R$, let $%
[x,y]:=xy-yx$ denote the additive commutator of $x$ and $y$. Also, let $A,B$ be subsets of $R$.  We denote $%
[A,B] $ (resp. $AB$) the additive subgroup of $R$ generated by all $%
[a,b]$ (resp. $ab$) for $a\in A$ and $b\in B$. Also, let ${\mathcal{I}}(A)$
be the ideal of $R$ generated by $A$.

An additive subgroup $L$ of $R$ is called a {\it Lie ideal} of $R$ if $%
[L,R]\subseteq L$. Given Lie ideals $L,K$ of $R$, it is known that both $%
[L,K]$ and $KL$ are Lie ideals of $R$. Recall that, given a subset $A$ of $R$, we denote
by $A^{+}$ the additive subgroup of $R$ generated by $A$.

\begin{lemma}\label{lem16}
Let $R$ a ring. Then $E(R)$ is a Lie ideal of $R$ and $[E(R), R]\subseteq U(R)^+$.
\end{lemma}

 \begin{proof}
 Indeed, let $e=e^{2}\in R$ and
$x\in R$. Then $e+ex(1-e)$ and $e+(1-e)xe$ are idempotents of $R$. So
\begin{equation*}
\lbrack e,x]=e+ex(1-e)-(e+(1-e)xe)\in E(R).
\end{equation*}%
Hence $[E(R),R]\subseteq E(R)$. Thus $E(R)$ is a Lie ideal of $R$.

On the other hand, $1+ex(1-e)$ and $1+(1-e)xe$
are units of $R$. Thus
$$
[e,x]=1+ex(1-e)-(1+(1-e)xe)\in U(R)^{+}.
$$
Hence $[E(R), R]\subseteq U(R)^+$.
\end{proof}

Clearly, if $X\subseteq Y\subseteq R$, then the $X$-semiprimeness of $R$
implies that $R$ is $Y$-semiprime. By Lemma \ref{lem16} and Theorem \ref{thm9}, we have the following

\begin{proposition}
\label{pro1} If a ring $R$ is $[E(R),R]$-semiprime, then $R$ is $E(R)$-semiprime
and hence is $U(R)$-semiprime.
\end{proposition}

A natural question is the following

\begin{problem}
\label{problem1}
Let $R$ be a semiprime ring. Find a sufficient and
necessary condition for $R$ to be $[E(R),R]$-semiprime.
\end{problem}

We note that Problem \ref{problem1} is completely answered by Theorem \ref%
{thm16} below. Motivated by Proposition \ref{pro1}, we proceed
to study the $L$-semiprimeness of a ring $R$, where $L$
is a Lie ideal of $R$. The first step is to establish a criterion for a ring
$R$ to be $L$-semiprime (see Theorem \ref{thm3}).

\begin{lemma}
\label{lem5} If $L$ is a Lie ideal of a ring $R$, then ${\mathcal{I}}([L,
L])\subseteq L+L^2$.
\end{lemma}

See, for instance, \cite[Lemma 2.1]{lee2022}) for its proof. Given a subset $%
A$ of a ring $R$, let $\widetilde{A}$ denote the subring of $R$ generated by
$A$. We continue with the following

\begin{lemma}
\label{lem1} Let $R$ be a semiprime ring with a Lie ideal $L$. Suppose that $%
aLb=0$, where $a,b\in R$. The following hold:

(i)\ $a{\mathcal{I}}([L, L])b=0$ and $a{\widetilde L}b=0$;

(ii)\ If $R$ be a prime ring and $[L, L]\ne0$, then either $a=0$ or $b=0$.
\end{lemma}

\begin{proof}
(i)\ Since $a[L,RaL]b=0$ and $aLb=0$, we get $aRaL^{2}b=0$. The
semiprimeness of $R$ implies $aL^{2}b=0$ and so $a(L+L^{2})b=0$. By Lemma %
\ref{lem5}, we get $a{\mathcal{I}}([L,L])b=0$. Note that $L+L^{2}$ is also a
Lie ideal of $R$. Repeating the same argument, we finally get $a{\widetilde{L%
}}b=0$ as ${\widetilde{L}}=\sum_{i=1}^{\infty }L^{i}$.

(ii)\ It follows directly from (i) and the primeness of $R$.
\end{proof}

The following gives a criterion for a ring $R$ to be $L$-semiprime.

\begin{theorem}
\label{thm3} Let $R$ be a ring, and let $L$ be a Lie ideal of $R$. Then:

(i)\ If $R$ is semiprime and ${\widetilde L}=R$, then $R$ is $L$-semiprime;

(ii)\ If $R$ is prime and $[L, L]\ne 0$, then $R$ is $L$-prime.
\end{theorem}

\begin{proof}
(i)\ Let $aLa=0$, where $a\in R$. By Lemma \ref{lem1} (i), $a{\widetilde L}%
a=0$ and so $aRa=0$. The semiprimeness of $R$ implies $a=0$.

(ii)\ Let $aLb=0$, where $a, b\in R$. By Lemma \ref{lem1} (ii), either $a=0$ or $b=0$, as desired.
\end{proof}

The following lemma is well-known. For the convenience of the reader, we
give its proof.

\begin{lemma}
\label{lem13} Let $R$ be a noncommutative prime ring, and let $I, J$ be
nonzero ideals of $R$. Then $\big[\lbrack I, J], [I, J]\big]\ne 0$.
\end{lemma}

\begin{proof}
Replacing $I, J$ by $I\cap J$, we may assume that $I=J$. By \cite[Lemma 1.5]%
{herstein1969} for the prime case, if $\big[a, [I, I]\big]=0$ where $a\in R$ then $%
a\in Z(R)$.

Suppose that $\big[\lbrack I, I], [I, I]\big]=0$. We have $[I, I]\subseteq
Z(R)$ and so $\big[\lbrack I, I], R\big]=0$. Thus $R\subseteq Z(R)$. That
is, $R$ is commutative, a contradiction.
\end{proof}

The following is a consequence of Theorem \ref{thm3} (ii) and
Lemma \ref{lem13}.

\begin{theorem}
\label{thm11}Let $R$ be a noncommutative prime ring. If $I$ and $J$ are nonzero
ideals of $R$, then $R$ is $[I, J]$-prime.
\end{theorem}

\begin{lemma}
\label{lem4} Let $R$ be a ring. We have

(i)\ If $A$ is an additive subgroup of $R$, then $[A, R]=[\widetilde{A}, R]$;

(ii)\ If $L$ is a Lie ideal of $R$, then $[\mathcal{I}([L, L]), R]\subseteq
[L, R]\subseteq L$.
\end{lemma}

\begin{proof}
(i)\ Let $a_1,\ldots,a_n\in A$ and $x\in R$. By induction on $n$, we get
\begin{equation*}
[a_1a_2\cdots a_n, x]=[a_2\cdots a_n, xa_1]+ [a_1, a_2\cdots a_nx]\in [A, R].
\end{equation*}
Thus $[A, R]=[\widetilde A, R]$.

(ii)\ By Lemma \ref{lem5} and (i), we have
\begin{equation*}
[\mathcal{I}([L, L]), R]\subseteq [L+L^2, R]\subseteq [\widetilde L, R]=[L,
R]\subseteq L,
\end{equation*}
as desired.
\end{proof}

The following is a slight generalization of \cite[Theorem 1.4]{lee2022} with $V=R$.

\begin{theorem}
\label{thm5} Let $R$ be a ring with a Lie ideal $L$. If ${\mathcal{I}}([L,
L])=R$, then $[L, R]=[R, R]$ and $R=[R, R]^2$.
In addition, if $R$ is semiprime, then it is $[L, R]$-semiprime.
\end{theorem}

\begin{proof}
By Lemmas \ref{lem5} and \ref{lem4} (i), we have
\begin{equation*}
[R, R]=[{\mathcal{I}}([L, L]), R]\subseteq [L+L^2, R]=[L, R].
\end{equation*}
Hence $[R, R]=[L, R]$.
Since ${\mathcal{I}}([L, L])=R$, it is clear that ${\mathcal{I}}([R, R])=R$ and so,
by \cite[Theorem 1.4]{lee2022}, $R=[R, R]+[R, R]^2$.
Thus we have
 \begin{eqnarray*}
[R, R]&=&\big[[R, R]+[R, R]^2, [R, R]+[R, R]^2\big]\\
         &=&\big[[R, R], [R, R]\big]+\big[[R, R]+[R, R]^2, [R, R]^2\big]\\
           &=&\big[[R, R], [R, R]\big]+\big[R, [R, R]^2\big]\\
         &\subseteq& [R, R]^2+[R, R]^2\\
         &=&[R, R]^2,
\end{eqnarray*}
where we have used the fact that $[R, R]^2$ is a Lie ideal of $R$.
Hence
$$
R=[R, R]+[R, R]^2\subseteq [R, R]^2+[R, R]^2=[R, R]^2,
$$
as desired.

In addition, assume that $R$ is semiprime. Since
$$
R=[R, R]^2=[L, R]^2\subseteq \widetilde{[L, R]},
$$
we have $R= \widetilde{[L, R]}$.
 In view of Theorem \ref{thm3}(i), $R$ is $[L, R]$-semiprime.
\end{proof}

\begin{remark}\label{remark6}
\normalfont
(i)\ In Theorem \ref{thm5}, $R$ is in general not $[L, L]$-semiprime. For instance, let $R:=\mathbb{M}_2(F)$, where $F$ is a field of characteristic $2$, and let $L:=[R, R]$. In view of \cite[Lemma 2.3]{lee2022}, we have $0\ne [L, L]\subseteq F$. Thus ${\mathcal{I}}([L, L])=R$ but $R$ is not $[L, L]$-semiprime as $R$ is not a domain.

(ii)\ In Theorem \ref{thm5},  if the assumption ${\mathcal{I}}([L, L])=R$ is replaced by ${\mathcal{I}}([L, R])=R$, we cannot conclude that
$R$ is $[L, R]$-semiprime.
See Remark \ref{remark1} (i) and (ii) in the next section, and a related result \cite[Theorem 1.1]{lee2022-1}.
\end{remark}

The equality $[L, L]=[L, R]$ is in general not true.
However, we always have the following

\begin{lemma}\label{lem17}
Let $R$ be a ring. Then $[E(R), R]=[E(R), E(R)]$.
\end{lemma}

\begin{proof}
Indeed, let $e=e^2\in R$
and $x\in R$. Then
$$
[e, x]=\big[e, [e, [e, x]]\big]\in \big[E(R), [E(R), R]\big]\subseteq [E(R), E(R)].
$$
So $[E(R), E(R)]=[E(R), R]$.
\end{proof}

\begin{theorem}
\label{thm7} Let $R$ be a prime ring with a nontrivial idempotent. Then $R$
is $[E(R), R]$-prime.
\end{theorem}

\begin{proof}
Let $E:=E(R)$. Clearly, we have $E\nsubseteq Z(R)$.
By Lemma \ref{lem17}, we have
$[E, E]=[E, R]\ne 0$. Let $I:=\mathcal{I}([E, E])$, a nonzero ideal of $R$.
It follows from Lemma \ref{lem4} (ii) that
\begin{equation*}
0\ne [I, R]=\emph{}[\mathcal{I}([E, E]), R]\subseteq [E, R]=[E, E].
\end{equation*}
By Lemma \ref{lem13}, $\big[[E, R], [E, R]\big]\ne 0$.
It follows from Theorem \ref{thm3} (ii) that $R$ is $[E, R]$-prime.
\end{proof}

Theorem \ref{thm7} is also an immediate consequence of Theorem \ref{thm16}.

\begin{corollary}
\label{cor5} Every prime $E(R)$-semiprime ring $R$ is either a domain or an $%
[E(R), R]$-prime ring.
\end{corollary}

\begin{proof}
Assume that $R$ is not a domain. Since $R$ is a prime ring, it contains
nonzero square zero elements. Let $a^{2}=0$, where $0\neq a\in R$. Since $R$ is
$E(R)$-semiprime, this
implies that $E(R)\neq \{0,1\}$. It follows from Theorem \ref{thm7} that $R$
is $[E(R),R]$-prime.
\end{proof}

We now apply Theorem \ref{thm5} to the case of matrix rings.

\begin{theorem}
\label{thm4} Let $R:=\mathbb{M}_{n}(A)$, where $A$ is a semiprime
 ring and $n>1$. Then $R$ is $[E(R), R]$-semiprime.
\end{theorem}

\begin{proof}
Let $E:=E(R)$. By \cite[Theorem 2.1]{lee2017}, we have ${\mathcal{I}}([E,
E])=R$. Since $A$ is a semiprime ring, so is $R$. By Theorem \ref{thm5}, $R$ is $[E, R]$-semiprime.
\end{proof}

The following corollary is a consequence of Theorem \ref{thm4} and
Proposition \ref{pro1}.

\begin{corollary}
\label{cor6} Let $R:=\mathbb{M}_{n}(A)$, where $A$ be a semiprime ring and $n>1$. Then $R$ is idempotent semiprime.
\end{corollary}

The above corollary is also a generalization of \cite[Theorem 10]{cal2}
which allows to prove easily the following result.

\begin{corollary}
\label{cor4}
The idempotent semiprime property does not pass to corners.
\end{corollary}

\begin{proof}
Choose a semiprime ring $R$, which is not idempotent semiprime, and $n>1$ an
integer. Then $e_{11}\mathbb{M}_{n}(R)e_{11}\cong R$ and so $\mathbb{M}%
_{n}(R)$ is idempotent semiprime but $R$ itself is not.
\end{proof}

\begin{corollary}
Idempotent semiprime is not a Morita invariant property of rings.
\end{corollary}

Throughout, we use the following notation. Let $R$ be a semiprime ring. We can define its {\it Martindale symmetric
ring of quotients} $Q(R)$. The center of this ring, denoted by $C$, is
called the {\it extended centroid} of $R$. It is known that $Q(R)$ itself
is a semiprime ring and $C$ is a regular self-injective ring. Moreover, $C$
is a field iff $R$ is a prime ring. We refer the reader to the book \cite%
{beidar1996} for details.

\section{Prime rings}
In this section, given a Lie ideal $L$ of a prime ring $R$, we obtain a complete
characterization for $R$ to be $L$-prime (see Theorem \ref{thm8}).
The following is a special case of \cite[Theorem]{her2}, which will be used
in the proofs below.

\begin{lemma}
\label{lem3} Let $R$ be a prime ring, $a,b\in R$ with $b\notin Z(R)$.
Suppose that $[a,[b,x]]=0$ for all $x\in R$. Then $a^{2}\in Z(R)$. In
addition, if $\text{\textrm{char}}\, R\ne 2$, then $a\in Z(R)$.
\end{lemma}

The following is well-known.

\begin{lemma}
\label{lem2}Let $R$ be a prime ring.
If $[a, R]\subseteq Z(R)$ where $a\in R$, then $a\in Z(R)$.
\end{lemma}

\begin{lemma}
\label{lem14} Let $R$ be a prime ring. If $a\in
RC\setminus C$, then $\dim_C[a, RC]>1$.
\end{lemma}

\begin{proof}
Suppose not. We have $\dim_C[a, RC]=1$ and $[a, RC]=Cw$ for some $w\in [a,
RC]$. In particular, $\big[w, [a, RC]\big]=0$. In view of Lemma \ref{lem3},
we have $w^2\in C$. Thus $w[a, RC]=Cw^2\subseteq C$.

Let $x\in RC$. Then $w[a, xa]=w[a, x]a\in C$. Since $w[a, x]\in C$, we get $%
w[a, x]=0$. Thus $w[a, R]=0$. The primeness of $R$ forces $w=0$ and so $a\in C$. This is a
contradiction.
\end{proof}

\noindent{\bf Definition.}\ A noncommutative prime ring $R$ is called {\it exceptional} if both $\text{\rm char}\,R= 2$ and $\dim_CRC=4$. Otherwise, $R$ is called
{\it non-exceptional}.\vskip6pt

A Lie ideal $L$ of a ring $R$ is called {\it proper} if $[I,R]\subseteq L$
for some nonzero ideal $I$ of $R$.
We need the following lemma (see  \cite[Lemma 7]{lanski1972}).

\begin{lemma}\label{lem18}
Let $R$ be a prime ring with a Lie ideal $L$. Then $L$ is noncentral iff $[L, L]\ne 0$ unless $R$ is exceptional.
\end{lemma}

Clearly, if $L$ is a nonzero central Lie ideal of $R$, then $R$ is $L$-prime if and only if $R$ is a domain.

\begin{theorem}
\label{thm8} Let $R$ be a prime ring, and let $L$
be a noncentral Lie ideal of $R$. If $R$ is not a domain, then $R$ is $L$%
-prime if and only if one of the following holds:

(i)\ $L$ is a proper Lie ideal of $R$;

(ii)\ $R$ is exceptional, $[L, L]=0$, $\dim_CLC=2$
and $LC=[a, RC]$, where $a\in L$ such that $a+\beta$ is invertible in $RC$
for all $\beta\in C$.
\end{theorem}

\begin{proof}
Clearly, $R$ is not commutative. By Lemma \ref{lem4} (ii), $[%
\mathcal{I}([L, L]), R]\subseteq L$.

``$\Longrightarrow$'':\ Suppose that $R$ is $L$-prime. If $\mathcal{I}%
([L,L])\neq 0$, then $L$ is a proper Lie ideal of $R$ and (i) holds. Assume
next that $\mathcal{I}([L,L])=0$, that is, $[L,L]=0$. Since $L$ is a
noncentral Lie ideal of $R$, it follows from Lemma \ref{lem18} that
$R$ is exceptional. Also, $R$ is not a domain
and so $RC\cong \mathbb{M}_{2}(C)$. Clearly, $LC$ is a commutative Lie ideal
of $RC$.

Since $LC$ is noncentral, $0\ne [LC, RC]\subseteq LC$ and, by Lemma \ref{lem2}, $[LC, RC]\nsubseteq C$.
Choose a nonzero
element $a\in [LC, RC]\setminus C$. Then $0\ne [a, RC]\subseteq LC$. It is well-known
that $Z(R)\ne 0$ and $C$ is the quotient
field of $Z(R)$. Hence we may choose $a\in L$. It follows from Lemma \ref%
{lem14} that $\dim_C [a, RC]>1$.

Suppose that $\dim_CLC=3$. Then $RC=LC+Cz$ for some $z\in RC$. Thus
\begin{equation*}
[RC, RC]=[LC+Cz, LC+Cz]=[LC, Cz]\subseteq LC.
\end{equation*}
In particular, $\big[\lbrack RC, RC], [RC, RC]\big]=0$, a contradiction (see
Lemma \ref{lem13}). Hence $LC=[a, RC]$ and $\dim_C LC=2$.

Suppose on the contrary that $a+\beta$ is not invertible in $RC$ for some $%
\beta\in C$. Then we can choose nonzero $b, c\in RC$ such that $b(a+\beta)=0$
and $(a+\beta)c=0$. We may choose $b, c\in R$. Given $x\in RC$, we have
\begin{equation*}
b[a, x]c=b[a+\beta, x]c=b(a+\beta)xc - bx(a+\beta)c=0.
\end{equation*}
Hence $b[a, RC]c=0$. That is, $bLCc=0$ and so $bLc=0$. This is a
contradiction as $R$ is $L$-prime. So (ii) holds.

``$\Longleftarrow$'':\ Suppose that (i) holds. Then $[I,R]\subseteq L$ for some
nonzero ideal $I$ of $R$. In view of Theorem \ref{thm11}, $R$ is $[I,R]$%
-prime and so it is $L$-prime. We next consider the case (ii). Let $b,c\in R$
be such that $bLc=0$. Then $bLCc=0$ and so $b[a,x]c=0$ for all $x\in RC$.
That is,
\begin{equation*}
baxc-bxac=0
\end{equation*}%
for all $x\in RC$. Suppose first that $ba$ and $b$ are $C$-independent. In
view of \cite[Theorem 2]{martindale}, $c=0$ follows, as desired. Suppose next that $%
ba$ and $b$ are $C$-dependent. That is, there exists $\beta \in C$ such that
$ba=\beta b$. So $b(a-\beta )=0$. Since, by assumption, $a-\beta $ is
invertible in $RC$ and we get $b=0$. Hence $R$ is $L$-prime.
\end{proof}

\begin{remark}
\normalfont\label{remark10} The case (ii) of Theorem \ref{thm8} indeed
occurs. Let $R:=\mathbb{M}_{2}(F)$, where $F$ is a field of characteristic $%
2 $.

(i)\ Let
\begin{equation*}
L=\{\left[
\begin{array}{cc}
\alpha & \beta \\
\beta & \alpha%
\end{array}%
\right] \mid \alpha ,\beta \in F\}=[a,R],
\end{equation*}%
where $a:=\left[
\begin{array}{cc}
1 & 1 \\
1 & 1%
\end{array}%
\right] $. Then $L$ is a noncentral Lie ideal of $R$ satisfying $[L,L]=0$.
Since $a\in L$ and $a^{2}=0$, we get $aLa=0$ but $a\neq 0$. Thus, $R$ is not
$L$-prime.

(ii)\ We choose $F$ such that there exists $\eta\in F$ satisfying $%
\eta\notin F^{(2)}:=\{\mu^2\mid \mu\in F\}$. Let
\begin{equation*}
L=\{\left[
\begin{array}{cc}
\alpha & \beta \\
\beta\eta & \alpha%
\end{array}%
\right]\mid \alpha ,\beta \in F\}=[a, R],
\end{equation*}
where $a:=\left[
\begin{array}{cc}
1 & 1 \\
\eta & 1%
\end{array}%
\right]$. Then $L$ is a noncentral Lie ideal of $R$ satisfying $[L,L]=0$.
Note that $a+\beta$ is invertible for all $\beta\in F$. Hence $R$ is $L$%
-prime.
\end{remark}

We follow the notation given in \cite{C}. A subset $X$ of a ring $R$ is said
to be invariant under {\it special} automorphisms if $(1+t)X(1+t)^{-1}%
\subseteq X$ for all $t\in R$ such that $t^{2}=0$. Clearly, if $X\subseteq R$
is invariant under special automorphisms, then so is $X^{+}$. Also, $R$ is $%
X $-prime if and only if it is $X^{+}$-prime.

\begin{theorem}
\label{thm10} Let $R$ be a prime ring with a nontrivial idempotent and let $%
X $ be a subset of $R$ invariant under special automorphisms. If $%
X\not\subseteq Z(R)$, then $R$ is $X$-prime.
\end{theorem}

\begin{proof}
Without loss of generality we can replace $X$ by $X^+$ and assume that $X$
is an additive subgroup of $R$. Clearly, $R$ is not commutative.

Case 1:\ $R$ is non-exceptional. Then, as $%
X\nsubseteq Z(R)$, we can apply \cite[Theorem 1]{C} to get that $X$ contains
a proper Lie ideal $L$ of $R$. In view of Theorem \ref{thm8} (i), $R$ is $L$%
-prime and hence is $X$-prime.

Case 2:\ $R$ is exceptional. It follows from
\cite[Lemma 11]{C} that $XZ(R)$ contains a proper Lie ideal of $R$. Thus,
using Theorem \ref{thm8} (i) again, we obtain that $R$ is $XZ(R)$-prime and
hence is $X$-prime.
\end{proof}

The following are natural examples of $X$:\ potent elements, potent elements
of a fixed degree, nilpotent elements, nilpotent elements of a fixed degree
(in particular, elements of square zero), $[E(R),R]$, $U(R)$, $E(R)$ etc.
Moreover, if $A,B$ are invariant under special automorphisms, then so are $%
AB $ and $[A,B]$.

Let $N(R)$ denote the set of all nilpotent elements of $R$.
We end this section with the following corollary, which is a consequence of Theorem \ref{thm10}.

\begin{corollary}\label{cor7}
Let $R$ be a ring possessing a nontrivial idempotents. The following are equivalent:

(i)\ $R$ is a prime ring;

(ii)\ $R$ is $U(R)$-prime;

(iii)\ $R$ is $N(R)$-prime.
\end{corollary}

\section{Derivations}
By a {\it derivation} of a ring $R$ we mean an additive map $d\colon R\to R$
satisfying
$$
d(xy)=d(x)y+xd(y)
$$
for all $x, y\in R$. A derivation $d$ of $R$ is called {\it inner} if there exists $b\in R$ such that
$d(x)=[b, x]$ for all $x\in R$. In this case, we denote $d=\text{\rm ad}_b$.
Otherwise, $d$ is called {\it outer}. In this section we characterize the $d(R)$-semiprimeness of a given prime ring $R$. Moreover, we also study the
$d(L)$-semiprimeness of $R$ for $L$ a Lie ideal of $R$.

Let $R$ be a prime ring with a derivation $d$. It is known that $d$ can be uniquely extended to a derivation, denoted by $d$ also, of $Q(R)$ (by applying a standard argument).
We say that $d$ is {\it $\mathfrak{X}$-inner} if  $d$ is inner on $Q(R)$ and $d$ is called {\it $\mathfrak{X}$-outer}, otherwise.

We need a preliminary proposition due to Kharchenko (see \cite[Lemma 2]{kharchenko1978}).\vskip 4pt

\begin{proposition}\label{pro3}
Let $R$ be a prime ring with a derivation $\delta$.
Suppose that there exist finitely many $a_i, b_i, c_j,
d_j\in Q(R)$ such that
$$
\sum_ia_i\delta(x)b_i+\sum_jc_jxd_j=0
$$
for all $x\in R$. If $\delta$ is $\mathfrak{X}$-outer, then
$\sum_ia_ixb_i=0=\sum_jc_jxd_j$ for all $x\in R$.
\end{proposition}

 Given $b\in Q(R)$, let $\ell_R(b):=\{x\in R\mid xb=0\}$, the left annihilator of $b$ in $R$.
 Similarly, we denote by $r_R(b)$ the right annihilator of $b$ in $R$.

 We are now ready to prove the first main theorem in this section.

\begin{theorem}\label{thm21}
Let $R$ be a noncommutative prime ring with a derivation $d$. Then $R$ is $d(R)$-semiprime iff one of the following conditions holds

(i)\ $d$ is $\mathfrak{X}$-outer;

(ii)\ $d=\text{\rm ad}_b$ for some $b\in Q(R)$, and for any $\beta\in C$, either $\ell_R(b+\beta)=0$ or $r_R(b+\beta)=0$.
\end{theorem}

\begin{proof}
``$\Longrightarrow$":\  Suppose that $R$ is $d(R)$-semiprimeness. Assume that $d$ is $\mathfrak{X}$-inner.
 Thus there exists $b\in Q(R)$ such that $d(x)=[b, x]$ for all $x\in R$. We claim that given any $\beta\in C$,  either $\ell_R(b+\beta)=0$ or $r_R(b+\beta)=0$. Otherwise, there exist $\beta\in C$ and nonzero elements $a, c\in R$ such that
 $a(b+\beta)=0=(b+\beta)c$. By the primeness of $R$, $w:=cya\ne 0$ for some $y\in R$. Then
 $$
 wd(z)w=w[b+\beta, z]w=cya(b+\beta)zw-wz(b+\beta)cya=0
 $$
 for all $z\in R$. That is, $wd(R)w=0$ with $w\ne 0$. So $R$ is not $d(R)$-semiprimeness, a contradiction.

 ``$\Longleftarrow$":\  (i)\ Assume that $d$ is $\mathfrak{X}$-outer. Let $ad(x)a=0$ for all $x\in R$.
By Proposition \ref{pro3}, we get $aya=0$ for all $y\in R$. The primeness of $R$ implies $a=0$.
This proves that $R$ is $d(R)$-semiprime.

(ii)\ Let $ad(x)a=0$ for all $x\in R$. Since $d(x)=[b, x]$ for all $x\in R$, we get
$$
0=ad(x)a=abxa-axba
$$
 for all $x\in R$.
In view of \cite[Theorem 2]{martindale}, there exists $\beta\in C$ such that $ab=-\beta a$, i.e. $a\in \ell_R(b+\beta)$.
Thus $ax(b+\beta)a=0$ for all $x\in R$. Hence $(b+\beta)a=0$, i.e., $a\in r_R(b+\beta)$.
Since either $\ell_R(b+\beta)=0$ or $r_R(b+\beta)=0$, we get $a=0$.
Hence $R$ is $d(R)$-semiprime.
\end{proof}

As a direct application of the above theorem we get the following

\begin{corollary}
\normalfont\label{cor15}
Let $R:=\mathbb{M}_m(D)$, where $D$ is a noncommutative division ring, $m\geq 1$, and let $d=\text{\rm ad}_b$, where
$b:=\sum_{i=1}^m\mu_ie_{ii}$, where $\mu_i\in D$ for all $i$.
Then $R$ is $d(R)$-semiprime iff $\mu_i\notin Z(D)$ for any $i$.
\end{corollary}

Let $x\in \mathbb{M}_m(F)$, where $F$ is a field. We denote by $\text{\rm det}\,(x)$ the determinant of $x$.

\begin{corollary}\label{cor2}
Let $R:=\mathbb{M}_m(F)$, where $F$ is a field, $m>1$, and let $d$ be a derivation of $R$.
Then $R$ is $d(R)$-semiprime iff either $d$ is outer or $d=\text{\rm ad}_b$ for some $b\in R$ such that $\text{\rm det}\,(b+\beta)\ne 0$ for any $\beta\in F$.

In addition, if $F$ is algebraically closed, then $R$ is $d(R)$-semiprime iff $d$ is outer.
\end{corollary}

In Corollary \ref{cor2}, let $\overline F$ be the algebraic closure of the field $F$.
It is known that the matrix $b$ can be upper triangularizable in $\mathbb{M}_m(\overline F)$, that is,
there exists a unit $u$ of $\mathbb{M}_m(\overline F)$ such that
$ubu^{-1}=\sum_{1\leq i\leq j\leq n}\mu_{ij}e_{ij}$,
where $\mu_{ij}\in \overline F$. Thus $\text{\rm det}\,(b+\beta)\ne 0$ for any $\beta\in F$
iff $\mu_{ii}\in {\overline F}\setminus F$ for all $i$.\vskip4pt

Motivated by the above two results, it is natural to raise the following

\begin{problem}\label{problem5}
Let $R:=\mathbb{M}_m(D)$, where $D$ is a noncommutative division ring, $m\geq 1$. Characterize elements $b\in R$ such that
$b+\beta\in U(R)$ for any $\beta\in Z(D)$.
\end{problem}

The problem seems to be related to the triangularizability of the element $b$.
We next deal with the $d(L)$-semiprimeness of a prime ring $R$, where $L$ is a Lie ideal of $R$ and $d$ is a derivation of $R$.

\begin{lemma}\label{lem19}
Let $R$ be a prime ring with a nonzero ideal $I$, and $a_i, b_i\in Q(R)$ for $i=1,\ldots,m$. Then:

(i)\ $\sum_{i=1}^ma_ixb_i=0$ for all $x\in I$ iff $\sum_{i=1}^mb_ixa_i=0$ for all $x\in I$;\vskip4pt

(ii)\ $\sum_{i=1}^ma_iwb_i=0$ for all $w\in [I, I]$ iff $\sum_{i=1}^mb_ixa_i\in C$ for all $x\in R$.
\end{lemma}

\begin{proof}
(i)\ If follows directly from \cite[Corollary 2.2]{lee2004}.

(ii)\ Applying (i) we have
\begin{eqnarray*}
\sum_{i=1}^ma_i[x, y]b_i=0\ \forall x, y\in I\ \ &\Leftrightarrow&\ \sum_{i=1}^ma_ix(yb_i)-(a_iy)xb_i=0\ \forall x, y\in I\\
&\Leftrightarrow&\sum_{i=1}^m(yb_i)xa_i-b_ix(a_iy)=0\ \forall x, y\in I\\
&\Leftrightarrow&\big[y, \sum_{i=1}^mb_ixa_i\big]=0\ \forall x, y\in I\\
&\Leftrightarrow&\sum_{i=1}^mb_ixa_i\in C\ \forall x\in I\\
&\Leftrightarrow&\sum_{i=1}^mb_ixa_i\in C\ \forall x\in R,
\end{eqnarray*}
where the last equivalence holds as $R$ and $I$ satisfy the same GPIs with coefficients in $Q(R)$ (see \cite[Theorem 2]{chuang1988}).
\end{proof}

The following is the second main result in this section.

\begin{theorem}\label{thm23}
Let $R$ be a non-exceptional prime ring, not a domain, with a noncentral Lie ideal $L$, and let $d$ be a derivation of $R$.
Then:

(i)\ If $d$ is $\mathfrak{X}$-outer, then $R$ is $d(L)$-semiprime;

(ii)\ If $d$ is $\mathfrak{X}$-inner, then $R$ is $d(L)$-semiprime iff it is $d(R)$-semiprime.
\end{theorem}

\begin{proof} Since $R$ is non-exceptional, either $\text{\rm char}\,R\ne 2$ or $\dim_CRC>4$. In view of Lemma \ref{lem18}, we have $[L, L]\ne 0$. Thus, by Theorem \ref{thm3} (ii), $R$ is $L$-semiprime.

(i)\ Assume that $ad(L)a=0$, where $a\in R$.
Let $x\in L$ and $r\in R$. Then $ad([x, r])a=0$ and so
\begin{equation*}
a\big([d(x), r]+[x, d(r)]\big)a=0.
 \end{equation*}
Since $d$ is $\mathfrak{X}$-outer, applying Proposition \ref{pro3} we get
\begin{equation*}
a\big([d(x), r]+[x, z]\big)a=0
\end{equation*}
 for all $x\in L$ and all $r, z\in R$. In particular, $a[x, z]a=0$ for all $x\in L$ and $z\in R$.
 That is, $a[L, R]a=0$.

In particular, $a[L, RaR]a=0$. Since $a[L, R](aR)a=0$, we get $aR[L, aR]a=0$. By the primeness of $R$, we have $[L, aR]a=0$
and so $[L, a]Ra=0$. Thus $[a, L]=0$. By the fact that $a[L, R]a=0$ and $[L, R]\subseteq L$, we have $a^2[L, R]=0$.
This implies $a^2=0$ as $L$ is noncentral.

It follows from $a[L, aR]a=0$ and $a^2=0$ that $aLaRa=0$, implying $aLa=0$. Since $R$ is $L$-semiprime, we get $a=0$, as desired.
\vskip4pt

(ii)\ Assume that $d$ is $\mathfrak{X}$-inner.
Clearly, if $R$ is $d(L)$-semiprime, then it is $d(R)$-semiprime.
Conversely, assume that $R$ is $d(R)$-semiprime.
Since $d$ is $\mathfrak{X}$-inner, there exists $b\in Q(R)$ such that $d(x)=[b, x]$ for all $x\in R$. In view of Lemma \ref{lem18}, $[L, L]\ne 0$.
Hence, by Lemma \ref{lem4} (i), we have $0\ne [K, R]\subseteq L$, where $K:=\mathcal{I}([L, L])$.

Let $ad(L)a=0$, where $a\in R$. The aim is to prove $a=0$. Then
$a[b, x]a=0$ and so
\begin{equation}\label{eq:11}
abxa=axba
\end{equation}
for all $x\in [K, K]\subseteq L$. Since $R$ is non-exceptional, the proof is divided into the following two cases.

Case 1:\ $\dim_RC>4$.
Applying Lemma \ref{lem19} (ii) to Eq.\eqref{eq:11}, we have
\begin{equation}\label{eq:12}
ayab-baya\in C
\end{equation}
for all $y\in R$.

Suppose first that $ay_0ab-bay_0a\ne 0$ for some $y_0\in R$. Applying \cite[Fact 3.1]{CL} to Eq.\eqref{eq:12}, we get $Q(R)=RC$ and $\dim_CRC<\infty$.
It follows from \cite[Theorem 1.1]{CL} that $\dim_CRC\leq 4$, a contradiction.
Thus $ayab-baya=0$
for all $y\in R$. In view of Lemma \ref{lem19} (i), $abya-ayba=0$
for all $y\in R$. That is, $ad(R)a=0$. Since $R$ is $d(R)$-semiprime, we get $a=0$, as desired.

Case 2:\ $\dim_RRC=4$ and $\text{\rm char}\,R\ne 2$. Since $R$ is not a domain, we have $RC\cong \mathbb{M}_2(C)$.
Note that $KC=RC$ in this case. Moreover, $[KC, KC]+C=RC$ as  $\text{\rm char}\,R\ne 2$. Clearly, Eq.\eqref{eq:11}
holds for all $x\in [KC, KC]$.

Let $y\in RC$. Then $y=x+\beta$ for some $x\in [KC, KC]$ and $\beta\in C$. By Eq.\eqref{eq:11} we have
$$
abya=abxa+ab\beta a=axba+\beta aba=a(x+\beta)ba=ayba.
$$
Thus $a[b, y]a=0$ for all $y\in RC$. In particular, $ad(R)a=0$. Since $R$ is $d(R)$-semiprime, we get $a=0$, as desired.
\end{proof}

By Theorems \ref{thm23} and \ref{thm21}, we have the following

\begin{corollary}\label{cor8}
Let $R$ be a non-exceptional prime ring, not a domain, with a noncentral Lie ideal $L$, and let $d$ be a derivation of $R$.
Then $R$ is $d(L)$-semiprime if one of the following conditions holds:

(i)\ $d$ is $\mathfrak{X}$-outer;

(ii)\ There exists $b\in Q(R)$ such that $d(x)=[b, x]$ for all $x\in R$. Moreover, given any $\beta\in C$,
either $\ell_R(b+\beta)=0$ or $r_R(b+\beta)=0$.
\end{corollary}

\begin{example}\label{example4}
\normalfont Let $R:=\mathbb{M}_{2}(F)$, where $F$ is a field of characteristic $2$.
We choose $F$ such that there exists $\eta\in F$ satisfying $%
\eta\notin F^{(2)}:=\{\mu^2\mid \mu\in F\}$.
 Let
\begin{equation*}
L=\{\left[
\begin{array}{cc}
\alpha & \beta \\
\beta & \alpha%
\end{array}%
\right] \mid \alpha ,\beta \in F\}=[a,R],
\end{equation*}%
where $a:=\left[
\begin{array}{cc}
1 & 1 \\
1 & 1%
\end{array}%
\right] $. Then $L$ is a noncentral Lie ideal of $R$ satisfying $[L,L]=0$.
Since $a\in L$ and $a^{2}=0$, we get $aLa=0$ but $a\neq 0$. Thus, $R$ is not
$L$-prime.
Let $d(x):=[b, x]$ for all $x\in R$, where
$b:=\left[
\begin{array}{cc}
1 & 1 \\
\eta & 1%
\end{array}%
\right].$ Then $d$ is an inner derivation of $R$.
Since $d(L)\subseteq L$ and $R$ is not $L$-semiprime, it follows that $R$ is not $d(L)$-semiprime.
Notice that $b+\beta$ is a unit of $R$ for all $\beta\in F$. In view of Theorem \ref{thm21} (ii), $R$ is $d(R)$-semiprime.
\end{example}

\section{Semiprime rings I:\ main results}
Recall that, given a semiprime ring $R$, we denote by $Q(R)$ the Martindale symmetric
ring of quotients of $R$ and by $C$ the extended centroid of $R$.
The following two sections are about proving Theorems \ref{thm22} and \ref{thm13}.

\begin{theorem}
\label{thm22} Let $R$ be a semiprime ring, and let
$L$ be a Lie ideal of $R$. Suppose that $\ell _{R}([L,R])=0$. Then there
exists an idempotent $e\in C$ such that

(i)\ $ex^2\in C$ for all $x\in \widetilde L$, and

(ii)\ $(1-e)R$ is $(1-e)L$-semiprime.

\noindent In addition, if $R$ is $2$-torsion free, then $e=0$ and so $R$ is $L$-semiprime.
\end{theorem}

\begin{remark}\label{remark1}
\normalfont
In Theorem \ref{thm22}, since $\ell _{R}([L,R])=0$, $R$ is $L$-semiprime iff it is $[L, R]$-semiprime (see Corollary \ref{cor12} below).
However, we cannot conclude that $R$ is $[L, R]$-semiprime unless $R$ is $2$-torsion free (see Remark \ref{remark10} (i)).
\end{remark}

Applying  Corollary \ref{cor12} and the $2$-torsion free case of Theorem \ref{thm22}, we have the following

\begin{corollary}\label{cor14}
Let $R$ be a $2$-torsion free semiprime ring, and let
$L$ be a Lie ideal of $R$. Then $R$ is $[L, R]$-semiprime iff $\ell _{R}([L,R])=0$.
\end{corollary}

When $B$ is a subset of $\mathit{Id}(R)$
such that $B^+$ is a Lie ideal of $R$, we can get better conclusions for arbitrary semiprime rings
as follows.

\begin{theorem}
\label{thm13} Let $R$ be a semiprime ring, and let $B$ be a subset of $Id(R)$
such that $B^{+}$ is a Lie ideal of $R$. Then $\ell _{R}([B,R])=0$ iff $R$
is $[B,R]$-semiprime.
\end{theorem}

The meaning of the two results above deserves further understanding. Clearly, given a Lie ideal $L$ of a ring $R$,
the assumption $\ell_R(L)=0$ is
necessary to ensure the $L$-semiprimeness of $R$ but is in general not sufficient.
Thus Corollary \ref{cor14} and Theorem \ref{thm13} are indeed very special and have good conclusions.
It is natural to raise the following

\begin{problem}\label{problem2}
Let $R$ be a semiprime ring with a subset $X$, and let
$$
X^{(n)}:=\{x^n\mid x\in X\},
$$
 where $n$ is a positive integer.

(i)\ Characterize Lie ideals $L$ of $R$ such that $R$ is $L$-semiprime iff $\ell_R(L)=0$.

(ii)\ Find subsets $X$ of $R$  such that $R$ is $X$-semiprime iff $\ell_R(X)=0$.

(iii)\ Let $L$ be a Lie ideal of $R$. If $R$ is $L$-semiprime, is then it $L^{(n)}$-semiprime?

(iv)\ If $\ell_R([E(R), R])=0$, is then $R$ an $[E(R), R]^{(n)}$-semiprime ring?

(v)\ Let $R$ be $2$-torsion free and let $L$ be a Lie ideal of $R$. If $\ell_R(L, R])=0$, is then $R$ a $[L, R]^{(n)}$-semiprime ring?
\end{problem}

The following is clear.

\begin{lemma}\label{lem9}
Let $R$ be a ring with subsets $X, Y$. Then:

(i) If $R$ is both $X$-semiprime and $Y$-semiprime, then $R$ is $XY$-semiprime;

(ii) If $R$ is $X$-semiprime, then $R$ is $X^n$-semiprime for any positive integer $n$.
\end{lemma}

We give some examples of  Problem \ref{problem2} (ii).

\begin{example}\label{example6}
\normalfont Let $R$ be a semiprime ring, and let $\rho$ be a right ideal of $R$.

(i)\ The ring $R$ is $\rho^n$-semiprime iff $\ell_R(\rho)=0$, where $n$ is a positive integer.
Indeed, it is clear that $R$ is $\rho$-semiprime iff $\ell_R(\rho)=0$.  Thus, if $\ell_R(\rho)=0$, then
$R$ is $\rho$-semiprime and hence is $\rho^n$-semiprime (see Lemma \ref{lem9} (ii)). Conversely, assume that $f$
is $\rho^n$-semiprime. Then $\ell_R(\rho^n)=0$ and so $\ell_R(\rho)=0$ as $\rho^n\subseteq \rho$.

(ii)\ Let $X$ be a subset of $R$ such that, for any $x\in \rho$, $x^{n(x)}\in X$ for some positive integer $n(x)\leq m$, a fixed positive integer.
 Assume that $aXa=0$, where $a\in R$.  Then $ax^{n(x)}a=0$ for all $x\in R$.
 In view of \cite[Theorem 2]{lee1995}, we get $a\rho a=0$. Thus $R$ is $X$-semiprime iff $\ell_R(X)=0$ iff $\ell_R(\rho)=0$.

(iii)\ Let $X$ be a subset of $R$ such that, for any $x\in \rho$, $x^{n(x)}\in X$ for some positive integer $n(x)$. If $R$ has no  nonzero nil
one-sided ideals, then $R$ is $X$-semiprime iff $\ell_R(\rho)=0$ (see \cite[Theorem 1]{lee1995}).
\end{example}

\section{Semiprime rings II:\ proofs}

We begin with some preliminaries.
Given an ideal $I$ of $Q(R)$, it follows from the semiprimeness of $Q(R)$
that, for $a\in Q(R)$, $aI=0$ iff $Ia=0$. Thus $\ell _{Q(R)}(I)$, the left
annihilator of $I$ in $Q(R)$, is an ideal of $Q(R)$.

By an {\it annihilator ideal}
of $Q(R)$, we mean an ideal $N$ of $Q(R)$ such that $N=\ell_{Q(R)}(I)$ for some ideal $I$ of $Q(R)$.
The following is well-known (see, for instance, \cite[Lemma 2.10]{lee2017-1}).

\begin{lemma}
\label{lem15} Let $R$ be a semiprime ring. Then
every annihilator ideal of $Q(R)$ is of the form $eQ(R)$ for some idempotent
$e\in C$.
\end{lemma}

Let $R$ be a ring with a Lie ideal $L$. Given $x\in R$, $xLR\subseteq ([L,
x]+Lx)R\subseteq LR$. This proves that $LR$ is an ideal of $R$. This fact
will be used in the proof below.

\begin{lemma}
\label{lem6}Let $R$ be a semiprime ring, and let
$L$ be a Lie ideal of $R$. Then
\begin{equation*}
\ell_{Q(R)}(L)=\ell_{Q(R)}\big(Q(R)LQ(R)\big)=eQ(R)
\end{equation*}
for some idempotent $e\in C$.
\end{lemma}

\begin{proof}
Let $Q:=Q(R)$. For $a,b\in Q$, it is easy to prove that $aRb=0$ iff $aQb=0$
(it follows from the definition of $Q$). It suffices to claim that, given $%
a\in Q$, $aL=0$ iff $aQLQ=0$. The converse implication is trivial. Suppose
that $aL=0$. Then $aLR=0$. Note that $LR$ is an ideal of the semiprime ring $%
R$. We get $LRa=0$ and hence $LQa=0$. So $QLQa=0$ and the semiprimeness of $%
Q $ forces $aQLQ=0$. It follows from Lemma \ref{lem15} that
\begin{equation*}
\ell _{Q}(L)=\ell _{Q}(QLQ)=eQ
\end{equation*}%
for some idempotent $e\in C$.
\end{proof}

\begin{lemma}
\label{lem8} Let $R$ be a semiprime ring with a Lie ideal $L$.
The following hold:

(i)\ If $aLa=0$
where $a\in R$, then $[a, L]=0$;

(ii)\ If $\ell_R([L, R])=0$, then given $a\in R$, $a[L, R]a=0$ implies $aLa=0$.
\end{lemma}

\begin{proof}
(i)\ Since $aLa=0$, we have $a[L, RaR]a=0$. By the fact that $a[L, R](aR)a=0$, we
get $aR[L, aR]a=0$. The semiprimeness of $R$ implies that $[L, aR]a=0$.
Since $a[L, R]a=0$, we have $[L, a]Ra=0$. It follows from the semiprimeness
of $R$ again that $[a, L]=0$, as desired.

(ii)\ Assume that $\ell_R([L, R])=0$ and $a[L, R]a=0$. Since $[L, R]$ is a
Lie ideal of $R$, it follows from (i) that $\big[a, [L, R]\big]=0$ and so $a^2[L, R]=0$. Hence $a^2=0$ as $\ell_R([L, R])=0$.
Now, it follows from $a[L, aR]a=0$ that $aLaRa=0$ and so, by the semiprimeness of $R$, $aLa=0$, as desired.
\end{proof}

As an immediate consequence of Lemma \ref{lem8} (ii), we have the following

\begin{corollary}\label{cor12}
Let $R$ be a ring with a Lie ideal $L$ satisfying $\ell_R([L, R])=0$. Then $R$ is $L$-semiprime iff it is $[L, R]$-semiprime.
\end{corollary}

Clearly, if $R$ is a prime ring and $a[b, R]=0$, where $a, b\in R$, then either $a=0$ or $b\in Z(R)$.
The following is a consequence of Corollary \ref{cor12}.

\begin{corollary}\label{cor13}
Let $R$ be a prime ring with a noncentral Lie ideal $L$. Then $R$ is $L$-semiprime iff it is $[L, R]$-semiprime.
\end{corollary}

The next aim is to study semiprime rings $R$ with a Lie ideal $L$
satisfying $\ell_R([L, R])=0$. We need a technical lemma.

\begin{lemma}
\label{lem7} Let $R$ be a semiprime ring with a Lie ideal $L$, and $n$ a
positive integer. Assume that $x^n\in Z(R)$ for all $x\in \widetilde L$.
Then ${\mathcal{I}}([L, L])\subseteq Z(R)$. In addition, if $R$ is $2$%
-torsion free, then $L\subseteq Z(R)$.
\end{lemma}

\begin{proof}
In view of Lemma \ref{lem5}, we have
${\mathcal{I}}([L, L])\subseteq L+L^2\subseteq \widetilde L$.
Thus $[x^n, R]=0$ for all $x\in {\mathcal{I}}([L, L])$. In view of \cite[%
Lemma 2]{lee1996}, we have $[x, R]=0$ for all $x\in {\mathcal{I}}([L, L])$.
That is, ${\mathcal{I}}([L, L])\subseteq Z(R)$. In particular, $[L, L]\subseteq
Z(R)$.

In addition, assume that $R$ is $2$-torsion free. Let $P$ be a prime ideal
of $R$ such that $\text{\textrm{char}}\,R/P\ne 2$. Working in $R:=R/P$, we
get $[\overline L, \overline L]\subseteq Z(\overline R)$ and so $\overline
L\subseteq Z(\overline R)$ (see \cite[Lemma 6]{bergen1981}). That is, $[L,
R]\subseteq P$. Since $R$ is $2$-torsion free, the intersection of prime
ideals $P$ of $R$ with $\text{\textrm{char}}\,R/P\ne 2$ is zero, it follows
that $[L, R]=0$ and so $L\subseteq Z(R)$, as desired.
\end{proof}

Let $R$ be a semiprime ring, and let $L$ be a Lie ideal of $R$. We denote $%
\ell _{R}(L):=\{a\in R\mid aL=0\}$, the left annihilator of $L$ in $R$. In
view of Lemma \ref{lem6},
\begin{equation*}
\ell _{R}(L)=R\cap \ell_{Q(R)}(L),
\end{equation*}
implying that $\ell _{R}(L)$ is an ideal of $R$.\vskip8pt

\noindent{\bf Proof of Theorem \ref{thm22}.}\vskip4pt

Let $Q:=Q(R)$. In view of Lemma \ref{lem6}, there exists an idempotent $e\in
C$ such that
\begin{equation*}
\ell_{Q}\Big(\sum_{x\in \widetilde L}R[x^2, R]R\Big)=eQ.
\end{equation*}
Clearly, $ex^2\in C$ for all $x\in \widetilde L$. This proves (i). We next
prove (ii), i.e., $(1-e)R$ is $(1-e)L$-semiprime.

Since $\ell_R([L, R])=0$, it is clear that
\begin{equation}  \label{eq:8}
\ell_{(1-e)R}([(1-e)L, (1-e)R])=0.
\end{equation}
Let $b(1-e)Lb=0$, where $b=(1-e)a\in (1-e)R$ for some $a\in R$. The aim is
to prove $b=0$. By $b\in (1-e)R$, we get $a(1-e)La=0$. Note that $(1-e)L$ is
a Lie ideal of the semiprime ring $(1-e)R$. In view of Lemma \ref{lem8} (i), we
get $[a, (1-e)L]=0$. In particular, $\big[a,[(1-e)L,R]\big]=0$.

Also, $0=a(1-e)La=a^{2}(1-e)L$, implying
\begin{equation*}
(1-e)a^{2}\big[(1-e)L, (1-e)R\big]=0.
\end{equation*}
It follows from Eq.\eqref{eq:8}  that $(1-e)a^{2}=0$, that is, $b^2=0$. By
Lemma \ref{lem4} (i), we have $[L, R]=[\widetilde L, R]$ and so $\big[a,
[(1-e)\widetilde L,R]\big]=0$, implying
\begin{equation}
\big[{(1-e)\widetilde L}, [b, (1-e)R]\big]=0.  \label{eq:3}
\end{equation}

Let $P$ be a prime ideal of $(1-e)R$ and let ${\overline{(1-e)R}}:= (1-e)R/P$%
. By Eq.\eqref{eq:3}, we have
\begin{equation*}
\big[{(1-e)\widetilde L}+P/P, [\overline b, {\overline{(1-e)R}}]\big]%
=\overline 0.
\end{equation*}

If $\overline b\notin Z( {\overline{(1-e)R}})$, it follows from Lemma \ref%
{lem3} that ${\overline x}^2\subseteq Z(\overline R)$ for all $x\in {%
(1-e)\widetilde L}$. Otherwise, $\overline b\in Z(\overline R)$ and so $%
\overline b=\overline 0$ as $b^2=0$. In either case, we have
\begin{equation*}
b[x^2, (1-e)R]\subseteq P
\end{equation*}
for all $x\in (1-e){\widetilde L}$. Since $P$ is arbitrary, the
semiprimeness of $(1-e)R$ forces $b[x^2, R]=0$ for all $x\in \widetilde L$.
Hence
\begin{equation*}
b\sum_{x\in \widetilde L}R[x^2, R]R=0.
\end{equation*}
Thus $b\in\ell_{Q}\Big(\sum_{x\in \widetilde L}R[x^2, R]R\Big)=eQ$. Since $%
b=(1-e)a$, we get $b=0$, as desired.\vskip4pt

Finally, assume that $R$ is $2$-torsion free. Since (i) holds, we have $%
ex^2\in C$ for all $x\in \widetilde L$. Note that $eL$ is a Lie ideal of the
$2$-torsion free semiprime ring $eR$. It follows from Lemma \ref{lem7} that $%
eL\subseteq C$. This implies that $e[L, R]=0$ and so $e=0$ as $\ell_R([L, R])=0$. Thus $R$ is $L$-semiprime.
\hfill $\square$
\vskip6pt

The following extends Theorem \ref{thm22} to the general case without the assumption $\ell_R([L, R])=0$.

\begin{theorem}
\label{thm19} Let $R$ be a semiprime ring, and
let $L$ be a Lie ideal of $R$. Then there exist orthogonal idempotents $%
e_{1}, e_{2}, e_{3}\in C$ with $e_{1}+e_{2}+e_{3}=1$ such that

(i)\ $e_1L\subseteq C$,

(ii)\ $e_2x^2\in C$ for all $x\in \widetilde L$, and

(iii)\ $e_3R$ is $e_3L$-semiprime.

\noindent In addition, if $R$ is $2$-torsion free, then $(e_1+e_2)L\subseteq
C$.
\end{theorem}

\begin{proof}
Let $Q:=Q(R)$. In view of Lemma \ref{lem6}, there exists an idempotent $%
e_1\in C$ such that
\begin{equation*}
\ell_{Q}([L, R])=e_1Q.
\end{equation*}
Clearly, we have $e_1L\subseteq C$. This proves (i). Let $f:=1-e_1$. Then $%
fQ $ is the Martindale symmetric ring of quotients of $fR$ (see \cite%
{beidar1996}). It is clear that $\ell_{fQ}([fL, fR])=0$.

In view of Theorem \ref{thm22}, there exists an idempotent $e_2\in fC$
such that $e_2x^2\in C$ for all $x\in \widetilde L$. This proves (ii).
Moreover, $e_3R$ is $e_3L$-semiprime, where $e_3:=f-e_2$ and hence (iii) is
proved.

Finally, assume that $R$ is $2$-torsion free. By Theorem \ref{thm22}, $%
e_2L\subseteq C$ and hence $(e_1+e_2)L\subseteq C$, as desired.
\end{proof}

In a prime ring $R$, a Lie ideal $L$ of $R$ is noncentral iff $\ell
_{R}([L,R])=0$. Thus we have the following (see also Theorem \ref{thm8} (i)).

\begin{corollary}
\label{cor10} Let $R$ be a prime ring of characteristic $\neq 2$, and let $L$
be a noncentral Lie ideal of $R$. Then $R$ is $L$-semiprime.
\end{corollary}

\begin{remark}
\label{remark4}\normalfont There exists a prime ring $R$ of characteristic $%
2 $ and a Lie ideal $L$ of $R$ such that $\ell _{R}([L,R])=0$ but $R$ is not
$L $-semiprime (see Remark \ref{remark1} (i)).
\end{remark}

\noindent{\bf Proof of Theorem \ref{thm13}.}\vskip4pt

``$\Longrightarrow$":\ Clearly, $[B, R]=[B^+, R]$. Let $a[B^+, R]a=0$, where
$a\in R$. The aim is to prove $a=0$.

Let $e\in B$ and $x\in R$. Then $ex(1-e)=[e,ex(1-e)]\in \lbrack B,R]$ and so
$aex(1-e)a=0$. Hence $aeR(1-e)a=0$ and the semiprimeness of $R$ implies $%
(1-e)ae=0$. Similarly, we have $ea(1-e)=0$ and so $[e, a]=0$. That is, $%
[a, B]=0$ and so $[a, B^{+}]=0$. Since $B^{+}$ is a Lie ideal of $R$, we get $%
\big[a, [B^{+}, R]\big]=0$ and hence
\begin{equation*}
\big[B^{+}, [a, R]\big]=0.
\end{equation*}%
Let $e\in B$, and let $P$ be a prime ideal of $R$. Then $\big[{\overline{e}}%
,[{\overline{a}}, {\overline{R}}]\big]=\overline{0}$, where ${\overline{R}}%
:=R/P$.

If ${\overline e}={\overline e}^2\notin Z({\overline R})$, by Lemma \ref{lem3} we get ${%
\overline a}\in Z({\overline R})$ and so $[a, R]\subseteq P$. Otherwise, we
have ${\overline e}\in Z({\overline R})$. That is, $[e, R]\subseteq P$. In
either case, we conclude that $[a, R][B, R]\subseteq P$. Since $P$ is an
arbitrary prime ideal of $R$, the semiprimeness of $R$ implies $[a, R][B,
R]=0$. Since $\ell_R([B, R])=0$, we get $[a, R]=0$ and so $a\in Z(R)$.

By the fact that $a[B, R]a=0$, we get $a^2[B, R]=0$ and so $a^2=0$. Since
the center of a semiprime ring is reduced, we have $a=0$, as desired.

``$\Longleftarrow$":\ Suppose not. Then $\ell_R([B, R])\ne 0$. Choose a
nonzero $a\in \ell_R([B, R])$. Then $a[B, R]a=0$, a contradiction.
\hfill $\square$
\vskip6pt

The following extends Theorem \ref{thm13} to the general case without the assumption $\ell_R([B, R])=0$.

\begin{theorem}
\label{thm20} Let $R$ be a semiprime ring, and
let $B$ be a subset of $\mathit{Id}(R)$ such that $B^+$ is a Lie ideal of $R$.
Then there exists an idempotent $e\in C$ such that

(i)\ $eB^+\subseteq C$, and

(ii)\ $(1-e)R$ is $(1-e)[B, R]$-semiprime.
\end{theorem}

\begin{proof}
Let $Q:=Q(R)$. It follows from Lemma \ref{lem6} that
\begin{equation*}
\ell _{Q}([B,R])=\ell _{Q}(Q[B^{+}, R]Q)=eQ
\end{equation*}%
for some $e=e^{2}\in C$. Then $[eB^{+}, R]=0$, implying $eB^{+}\subseteq C$
and this proves (i).

Note that $(1-e)B\subseteq \mathit{Id}((1-e)R)$ and $(1-e)B^{+}$ is a Lie
ideal of the semiprime ring $(1-e)R$. Let $a\in (1-e)R$ be such that $%
a[(1-e)B^{+}, (1-e)R]=0$. Then $a[B^{+}, R]=0$ and so $a\in eQ$. Hence $a=0$
follows. In view of Theorem \ref{thm13}, $(1-e)R$ is $(1-e)[B,R]$-semiprime
and this proves (ii).
\end{proof}

The following complements Proposition \ref{pro1} (see Problem \ref{problem1}).

\begin{theorem}
\label{thm16} Let $R$ be a semiprime ring. Then $\ell _{R}([E(R),R])=0$ iff $%
R $ is $[E(R),R]$-semiprime.
\end{theorem}

\begin{corollary}
\normalfont\label{cor1}
If $R$ is a semiprime ring such that $[E(R),R]$
contains a unit of $R$, then $R$ is $[E(R),R]$-semiprime.
\end{corollary}

\section{Regular rings}

When dealing with idempotent semiprime rings, it is natural to consider
rings having many idempotents. Since regular rings also have many
idempotents, a comparison is in order. Note that regular rings are
unit-semiprime (see \cite{cal2}).

\begin{theorem}
\label{thm1} Let $R$ be a semiprime ring. Suppose that, given any prime
ideal $P$ of $R$, either $R/P$ is a domain or there exists an idempotent $%
e\in R$ such that $\overline{e}\notin Z(\overline R)$, where $\overline R:=
R/P$. Then $R$ is idempotent semiprime.
\end{theorem}

\begin{proof}
We let $E:=E(R)$.
Let $aEa=0, $where $a\in R$. The aim is to prove $a=0$.

Let $P$ be a prime ideal of $R$. We have $\overline{a}%
\overline{E}\,\overline{a}=\overline{0}$ in $\overline R:= R/P$.  We divide the argument into two cases.

Case 1:\ $\overline R$ is a domain. Then $\overline{a}\overline{E}\,%
\overline{a}=\overline{0}$ and so $a\in P$.

Case 2:\ There exists an idempotent $e\in R$ such that $\overline{e}\notin
Z(\overline R)$. Then $\overline{a}\overline{E}\,\overline{a}=\overline{0}$
in $\overline R$. Note that $\overline{E}$ is also a Lie ideal of the prime ring $\overline R$.
By Lemma \ref{lem17} we
have
\begin{equation*}
[\overline{E}, \overline{R}]=\overline{[E, R]}=\overline{[E, E]}=[\overline{E}, \overline{E}].
\end{equation*}
Thus $[\overline{E}, \overline{E}]\ne 0$ since $\overline{e}\notin
Z(\overline R)$. By Theorem \ref{thm3} (ii), the prime ring $\overline R$ is $\overline E$-prime.
Thus $\overline{a}=0$
in $\overline R$ and so $a\in P$.

In either case, we have $a\in P$. Since $P$ is arbitrary, we get $a=0$. Thus
$R$ is idempotent semiprime.
\end{proof}

\begin{theorem}
\label{thm2} Every regular ring is idempotent semiprime.
\end{theorem}

\begin{proof}
Let $R$ be a regular ring and let $P$ be a prime ideal of $R$. Let $%
\overline R:=R/P$. Suppose that $\overline{e}\in Z(\overline R)$ for any
idempotent $e\in R$.

Let $a\in R$. Since $R$ is regular, $aba=a$ for some $b\in R$ and so $ab$ is
an idempotent of $R$. Hence $\overline{ab}\in Z(\overline R)$ and so $%
\overline{b}\overline{a}^{2}=\overline{a}$. Therefore $\overline R$ is a
reduced ring. By the primeness of $\overline R$, $\overline R$ is a domain.

 The above shows that every regular ring satisfies the assumptions of Theorem \ref{thm1}.
 Thus Theorem \ref{thm1} implies that $R$ is idempotent semiprime.
\end{proof}

\begin{example}
\normalfont\label{example3} (i)\ There exists a regular ring $R$ which is $%
E(R)$-semiprime but is not $[E(R),R]$-semiprime. For instance, let $%
R:=R_{1}\oplus R_{2}$, where $R_{1}=M_{n}(A) $ with $A$ a regular ring, $n>1$%
, and $R_{2}$ is a division ring. Clearly, $R$ is a regular ring. Then
\begin{equation*}
\ell _{R}([E(R),R])=\ell _{R_{1}}([E(R_{1}),R_{1}])\oplus \ell
_{R_{2}}([E(R_{2}),R_{2}])=0\oplus R_{2}\neq 0.
\end{equation*}%
In view of Theorem \ref{thm16}, $R$ is not $[E(R), R]$-semiprime. However, by Theorem \ref{thm2}, $R$
is $E(R)$-semiprime.

(ii)\ In the ring $R$ given in (i), it is easy to check that, given any
prime ideal $P$ of $R$, either $R/P$ is a domain or there exists an
idempotent $e\in R$ such that $\overline{e}\notin Z(\overline R)$, where $%
\overline R:= R/P$. Thus, in Theorem \ref{thm1}, we cannot conclude that $R$
is $[E(R),R]$-semiprime.
\end{example}

\section{Subdirect products}
We always assume that, given any ring $R$,
there exists a subset $X(R)$ associated with $R$. For instance, let $X(R)$
be $E(R)$, $U(R)$, $[E(R), R]$, $[R, R]$ etc. The following is clear.

\begin{theorem}
\label{thm110} Given rings $R_\beta$, $\beta\in J$, an index set, if $%
X(\prod_{\beta\in J}R_\beta)= \prod_{\beta\in J}X(R_\beta)$, then $%
\prod_{\beta\in J}R_\beta$ is $X(\prod_{\beta\in J}R_\beta)$-semiprime iff $%
R_\beta$ is $X(R_\beta)$-semiprime for all 1$\beta\in J$.
\end{theorem}

An idempotent semiprime ring $R$ just means that it is $\mathit{Id}(R)$%
-semiprime. The following is a direct consequence of of Theorem \ref{thm110}
by applying the property $Id(\prod_{\beta }R_{\beta })=\prod_{\beta
}Id(R_{\beta })$ for rings $R_{\beta }$.

\begin{proposition}
\label{pro4} A direct product of rings is idempotent semiprime iff each component
is idempotent semiprime.
\end{proposition}

\begin{theorem}
\label{thm14} Let $R$ be a semiprime ring. Suppose that, given any prime
ideal $P$ of $R$, there exists an idempotent $e\in R$ such that $\overline{e}%
\notin Z(\overline R)$, where $\overline R:= R/P$. Then the following hold:

(i)\ $R$ is $[E(R), R]$-semiprime;

(ii)\ $\overline R$ is $[E(\overline R), \overline R]$-prime for any prime
ideal $P$ of $R$, where $\overline R:= R/P$;

(iii)\ $R$ is a subdirect product of prime rings $R_\beta$, $\beta\in I$, an
index set, such that each $R_\beta$ is $[E(R_\beta), R_\beta]$-prime.
\end{theorem}

\begin{proof}
(i)\ Let $E:=E(R)$. By Theorem \ref{thm16}, it suffices to show that $\ell
_{R}([E,R])=0$. Suppose not, that is, $\ell _{R}([E,R])\neq 0$. There exists
a nonzero $a\in R$ such that $a[E,R]=0$. In particular, $a[E,R^{2}]=0$ and
so $aR[E,R]=0$. Since $R$ is a semiprime ring, $a\notin P$ for some prime
ideal $P$ of $R$. Then $aR[E,R]\subseteq P$, implying $[E,R]\subseteq P $.
Hence $\overline{e}\in Z(\overline{R})$ for any idempotent $e\in R$, where $%
\overline{R}:=R/P$, a contradiction.

(ii)\ In view of Theorem \ref{thm16}, it suffices to show that $\ell _{%
\overline{R}}([E(\overline{R}),\overline{R}])=\overline{0}$. Suppose that ${%
\overline{a}}[E(\overline{R}),\overline{R}]=\overline{0}$, where $\overline{a%
}\in \overline{R}$. Since $\overline{E(R)}\subseteq E(\overline{R})$, we get
${\overline{a}}[\overline{E(R)}, \overline{R}]=\overline{0}$ and so
${\overline{a}}\overline{R}[\overline{E(R)}, \overline{R}]=\overline{0}$. Since, by
assumption, $[\overline{E(R)}, \overline{R}]\neq \overline{0}$, it follows
from the primeness of $\overline{R}$ that ${\overline{a}}=\overline{0}$, as
desired. Thus $\overline{R}$ is $[E(\overline{R}), \overline{R}]$-semiprime.
The primeness of $\overline{R}$ implies that $\overline{R}$ is $[E(\overline{%
R}),\overline{R}]$-prime.

(iii)\ Since every semiprime ring is a subdirect product of prime rings,
(iii) follows directly from (ii).
\end{proof}

Notice that Theorem \ref{thm4} is also an immediate consequence of Theorem \ref%
{thm14} (i).

\begin{theorem}
\label{thm15} Let $R$ be a semiprime ring. Suppose that $\overline {E(R)}=E({%
\overline R})$ for any prime homomorphic image $\overline R$ of $R$. If $R$
is a subdirect product of $[E(R_\beta), R_\beta]$-prime rings $R_\beta$, $%
\beta\in I$, an index set, then $R$ itself is $[E(R), R]$-semiprime.
\end{theorem}

\begin{proof}
Write $R_{\beta }=R/P_{\beta }$, where $P_{\beta }$ is a prime ideal of $R$,
for each $\beta \in I$. By Theorem \ref{thm16}, it suffices to prove that $%
\ell _{R}([E(R),R])=0$. Otherwise, $a[E(R),R]=0$ for some nonzero $a\in R$.
Since $\bigcap_{\beta \in I}P_{\beta }=0$, there exists $\beta \in I$ such
that $\overline{a}\neq 0$ in $R_{\beta }=R/P_{\beta }$.

By assumption, we have $E(R_\beta)=\overline{E(R)}$. So ${\overline a}%
[E(R_\beta), R_\beta]=0$ and hence $\ell_{R_\beta}([E(R_\beta), R_\beta])\ne
0$. Theorem \ref{thm16} implies that $R_\beta$ is not $[E(R_\beta), R_\beta]$%
-prime, a contradiction.
\end{proof}

Motivated by Theorem \ref{thm15}, it is natural to ask the following\vskip6pt

\begin{problem}\label{problem4}
Characterize semiprime rings $R$ satisfying
the property that $\overline {E(R)}=E({\overline R})$ for any prime
homomorphic image $\overline R$ of $R$.
\end{problem}

Recall that if $A$ is an additive subgroup of a ring $R$, we say that
idempotents can be lifted modulo $A$ if, given $x\in R$ with $x-x^{2}\in A$, there exists $e=e^{2}\in R$ such that $e-x\in A$.

A ring $R$ is called {\it suitable} (or {\it exchange} \cite{nicholson1977})
if, given any $x\in R$, there exists $e=e^{2}\in R$ with $e-x\in R(x-x^{2})$.
Nicholson proved that a ring is suitable iff idempotents can be lifted
modulo every left ideal (see \cite[Corollary 1.3]{nicholson1977}). Hence we
have

\begin{proposition}
\label{pro2} Let $R$ be a suitable ring. Then $\overline {E(R)}=E({\overline
R})$ for any prime homomorphic image $\overline R$ of $R$.
\end{proposition}

The class of suitable rings is large:\ every homomorphic image of a suitable
ring, semiregular rings, clean rings and many others (see \cite%
{nicholson1977}).

\begin{theorem}
\label{thm17} Let $R$ be a semiprime suitable ring. Assume that $R/P$ contains a nontrivial idempotent for any prime ideal $P$
of $R$. Then $R$ is $[E(R), R]$-semiprime.
\end{theorem}

\begin{proof}
Since $R$ is a semiprime ring, $R$ is a subdirect product of prime rings $%
R_\beta$, $\beta\in I$, an index set. By the fact that $R$ is suitable, it
follows from Proposition \ref{pro2} that $[E(R),
R]+P_\beta/P_\beta=[E(R_\beta), R_\beta]$ for any $\beta\in I$. In view of
Theorem \ref{thm7}, every $R_\beta$ is $[E(R_\beta), R_\beta]$-prime. Hence,
by Theorem \ref{thm15}, $R$ is $[E(R), R]$-semiprime.
\end{proof}

\end{document}